\documentclass[12pt]{article}
\usepackage[russian, english]{babel}
\usepackage {amsmath}
\usepackage {amsthm}
\usepackage {amssymb}
\usepackage{wrapfig}
\usepackage{amsfonts}
\usepackage{textcomp}
\usepackage{graphicx}
\usepackage{float}
\usepackage{caption}
\usepackage{algorithm}
\usepackage{algpseudocode}
\usepackage{xcolor}
\usepackage{hyperref}
\usepackage{float}
\usepackage{mathenv}
\usepackage{booktabs}
\usepackage{indentfirst}
\usepackage{enumerate}

\newtheorem{Thm}{Theorem}[section]
\newtheorem{Lemm}{Lemma}[section]
\newtheorem{Rem}{Remark}[section]

\newtheorem{Co}{Corollary}[section]
\theoremstyle{definition}

\let\le=\leqslant
\let\ge=\geqslant
\let\leq=\leqslant
\let\geq=\geqslant

\begin{document}
\baselineskip 12pt

\begin{center}
\textbf{\Large On the Asymptotics of the Connectivity Probability of Erdos-Renyi Graphs} \\

\vspace{1.5cc}
{ B.B. Chinyaev$^{1}$, A.V. Shklyaev$^{2}$}\\
\vspace{0.3 cm}

{\small $^{1}$Lomonosov Moscow State University, bchinyaev.msu@gmail.com \\
 $^{2}$Lomonosov Moscow State University,  ashklyaev@gmail.com}
\end{center}
\vspace{1.5cc}

\begin{abstract}
In this paper, we investigate the exact asymptotic behavior of the connectivity probability in the Erdős–Rényi graph $G(n,p)$, under different asymptotic assumptions on the edge probability $p=p(n)$. We propose a novel approach based on the analysis of inhomogeneous random walks to derive this probability. We show that the problem of graph connectivity can be reduced to determining the probability that an inhomogeneous random walk with Poisson-distributed increments, conditioned to form a bridge, is actually an excursion.

\vspace{0.95cc}
\parbox{24cc}{{\it \textbf{Keywords:} Erdős–Rényi graph, connectivity probability, exact asymptotics, random walks}}
\end{abstract}

\section{Introduction} \label{intro}
The Erdős–Rényi random graph model was originally introduced in \cite{erdds1959random} and \cite{erdos1960evolution}. In this model, a graph \(G\) is considered with vertex set \(V=\{1,\dotsc, n\}\) and an adjacency matrix \(C\) whose entries \(c_{i,j}\) (for \(i<j\)) are independent and identically distributed Bernoulli random variables with parameter \(p=p(n)\).

A review of some of the results related to this model can be found in \cite{van2024random} and \cite{райгородский2010модели}. We are particularly interested in the asymptotic behavior of the connectivity probability \(P_n(p)\) of the graph as \(n\to\infty\) and \(p(n)\to 0\). We now recall several known results (see \cite{stepanov1970probability}) on this problem:

\begin{enumerate}[1)]
\item Suppose that $p(n)=(\ln n + \alpha + o(1))/n$, with $\alpha>0$. Then
\begin{equation}\label{res1}
P_n(p) = e^{-e^{-\alpha}} (1 + o(1)), \quad n\to\infty.    
\end{equation}
\item Suppose that $p(n)=c/n$, with $c>0$. Then 
\begin{equation}\label{res2}
P_n(p) = \left(1 - \frac{c}{e^c - 1}\right) \left(1-(1-c/n)^n\right)^n  (1 + o(1)), \quad n\to\infty.    
\end{equation}
\item Suppose that $p(n) = o(1/n^2)$ as $n\to \infty$. Then
\begin{equation}\label{res3}
P_n(p) = n^{n-2} p^{n-1} (1 + o(1)), \quad n\to\infty.    
\end{equation}
\end{enumerate}

The methods employed in these works rely on combinatorial estimates. In this paper, we propose a new approach for studying the connectivity probability of the Erdős–Rényi random graph. We show that the problem of determining the connectivity probability can be reduced to assessing whether a particular bridge, constructed using an inhomogeneous random walk, forms an excursion. Unfortunately, no convenient existing results of this type for inhomogeneous random walks are available in the literature, so we derive the necessary results independently.

It is worth noting that the obtained representation of the connectivity probability in terms of inhomogeneous random walks is non-asymptotic and uniformly applicable for any relation between \(p(n)\) and \(n\).

The paper is organized as follows. In Section~\ref{prelim}, we present some preliminary material. In particular, Section~\ref{general case} is devoted to the main lemma needed to derive the connectivity probability. Section~\ref{main theorem} states the main theorem. Sections~\ref{positivity edges}, \ref{Approximation lemm}, and \ref{positivity mid} contain the proofs of the necessary auxiliary lemmas, and Section~\ref{main proof} presents the proof of the main theorem.

\section{Preliminaries} \label{prelim}
In order to determine the connectivity probability, we need some additional constructions.

\subsection{Graph Exploration as a Random Walk}
To determine the connected component of a vertex \(v\) in a graph, a certain graph exploration process is used. In this process, the vertices can be \textit{active}, \textit{inactive}, or \textit{examined}. At the initial moment, the starting vertex \(v\) is designated as \textit{active}, while all other vertices are \textit{inactive}. Then, at each step, one \textit{active} vertex is selected (in the first step the starting vertex is chosen), and all of its inactive neighbors are added to the set of active vertices, while the vertex itself is moved into the set of examined vertices. The process continues as long as there remain active vertices, and the final set of examined vertices constitutes the connected component \(\mathcal{C}(v)\). The specific choice of the active vertex at each step is not essential (for instance, one may assume that the vertex which was added first to the active set is selected).

We consider this process (see also \cite{райгородский2010модели}, \cite{karp1990transitive}, \cite{nachmias2010critical}) in the random graph \(G(n,p)\). Let \(A_t\) denote the number of active vertices and \(U_t\) the number of inactive vertices at the beginning of step \(t\); denote by \(W_t\) the number of vertices that are transferred to the set of active vertices at that step, noting that the number of examined vertices coincides with the step number \(t\). We assume \(A_1 = 1, U_1 = n-1\), and hence
\[
A_{t+1} = A_t + W_t - 1, \quad U_{t+1} = U_{t} - W_t.
\]

Since the edges in the graph \(G(n,p)\) are independent, the random variables \(W_t\) at each step are binomially distributed:
\begin{equation*} 
	\mathbf{P}\left(W_t = k  \ | A_{t}=l,U_{t}=m\right) = 
	\left\{\begin{array}{ll}
            \binom{m}{k} p^k (1-p)^{m-k}, &  A_{t} > 0, \\
            0, &  A_{t} = 0.\\
        \end{array} \right.
\end{equation*}
For the graph to be connected, it is necessary that at each step (until step \(n\)) there remains at least one active vertex, i.e. 
\[
A_t = 1 + \left(\sum_{\tau=1}^t W_{\tau}\right) - t >0, \quad t<n.
\]

Consequently, the connectivity probability of the graph can be written in terms of this process as follows:
\begin{equation}\label{Pcon}
    P_n (p) = \sum\limits_{(j_{1} ,\dots , j_{n}) \in J_n} \ 
    \prod_{t=1}^{n}
    \left( \binom{n-1-j_{1}-\dotsb-j_{t-1}}{j_{t}} p^{j_{t}} (1-p)^{j_{t+1}+\dotsb+j_{n}} \right), 
\end{equation}
where
\begin{equation*}
J_n = \left\{(j_{1} ,\dots , j_{n}):  \sum_{i=1}^k j_i \geq k, \ k<n, \
                 \sum_{i=1}^{n} j_i = n-1\right\}.
\end{equation*}

In particular, we will consider the case \(p = p(n) \to 0\) as \(n \to \infty\), in which the expression (\ref{Pcon}) may become exponentially small. To find the asymptotics in this case, we will transform this expression into a more convenient form.

\subsection{Expression of Graph Connectivity via an Inhomogeneous Random Walk}\label{general case}
In this section, we reduce the problem of determining the connectivity of the graph $G(n,p)$ to the problem of the non-negativity of an inhomogeneous Poisson random walk conditioned to form a bridge. Unlike the expression in (\ref{Pcon}), which is formulated in terms of the positivity of dependent random variables, we consider a random walk with independent but non-identically distributed steps.

\begin{Lemm}\label{connect}
Let \(G(n,p)\) be an Erdős–Rényi graph. Then the connectivity probability is given by
$$
P_n (p) = \left(1-(1-p)^{n}\right)^{n-1} \mathbf{P}(S_k \geq 0, \ 0< k < n\mid S_{n} = -1 ), 
$$
where \(S_k = \sum_{i=1}^{k} X_i\) and the \(X_i\) are independent random variables such that \(X_i + 1\sim Poiss\left(\lambda_i\right)\), with 
 $$\lambda_i = \frac{np}{1-(1-p)^{n}}(1-p)^{(i-1)}.$$
\end{Lemm}
\begin{proof}
We transform the expression (\ref{Pcon}):
\begin{multline}\label{Connectivity1}
    P_n(p) = \sum\limits_{(j_{1} ,\dots , j_{n}) \in J_n} \
    \prod_{t=1}^{n}
    \left( \binom{n-1-j_{1}-\dotsb-j_{t-1}}{j_{t}} p^{j_{t}} (1-p)^{j_{t+1}+\dotsb+j_{n}} \right) =
\\
    = p^{n-1} (n-1)!
     \sum\limits_{(j_{1} ,\dots , j_{n}) \in J_n} \
    \prod_{t=1}^{n}
    \left( \frac{(1-p)^{(t-1)j_{t}}}{j_{t}!}
    \right).
\end{multline} 
We transform (with arbitrary $q>0$) the terms in the right-hand side of (\ref{Connectivity1}) into the form 
\begin{equation}
    \label{PoissonRepr}
    \exp\left(q\sum_{t=1}^{n} (1-p)^{t-1}\right) q^{-n+1}
    \prod_{t=1}^{n}
    \left( \exp\left(-q(1-p)^{t-1}\right) \frac{q^{j_t}(1-p)^{(t-1)j_{t}}}{j_{t}!}
    \right).
\end{equation}
Let \(X_t \sim Poiss(q(1-p)^{t-1})\); then 
$$
\exp\left(-q(1-p)^{t-1}\right) \frac{q^{j_t}(1-p)^{(t-1)j_{t}}}{j_{t}!} = \mathbf{P}(X_t = j_t).
$$
Set
$$q = \frac{np}{1-(1-p)^{n}} = n \left(\sum_{t=1}^{n} (1-p)^{t-1}\right)^{-1}.$$

\noindent Hence, the quantities in (\ref{PoissonRepr}) can be rewritten in the form
\begin{equation} \label{sum_term}
    \exp\left(n\right) 
    \left( \frac{1-(1-p)^{n}}{np}\right)^{n-1} 
    \
    \prod_{t=1}^{n}
    \left( \exp\left(-\lambda_t\right) \frac{\lambda_t^{j_{t}}}{j_{t}!}
    \right),
\end{equation}
\\
where \(\lambda_t = q(1-p)^{(t-1)}\). Substituting the expression (\ref{sum_term}) into (\ref{Connectivity1}), we obtain
$$  
    \left(1-(1-p)^{n}\right)^{n-1}
    \frac{\exp(n) (n-1)!}{n^{n-1} }
    \sum\limits_{(j_{1} ,\dots , j_{n}) \in J_n} \
    \prod_{t=1}^{n}
    \left( \exp\left(-\lambda_t\right) \frac{\lambda_t^{j_{t}}}{j_{t}!}
    \right).
$$
The resulting sum can be written as
$$
\mathbf{P}(S_k \geq 0, \ 0< k < n, \ S_{n} = -1 ),
$$
where \(S_k = \sum_{i=1}^{k} X_i\) and \(X_i + 1\sim Poiss\left(\lambda_i\right)\). It remains to note that
$$\mathbf{P}(S_{n} = -1) = \exp(-n) \frac{n^{n-1} }{(n-1)!} .$$ 
This completes the proof of Lemma \ref{connect}.
\end{proof}

The proven lemma allows us to study the connectivity probability of a graph for various parameters $p$. To do this, we need to compute the probability of the non-negativity of a random walk with independent and non-identical distributed steps, conditioned on returning to $-1$ at the end of the trajectory. An example of such a random walk $S_k$ is shown in Fig. \ref{ex1}.
It is important to note that the first step of the random walk $S_k$ has a positive mean, but with each subsequent step, this mean decreases, eventually becoming negative.

\begin{figure}[H]
    \centering
    \includegraphics[width=0.8\linewidth, height= 0.4\linewidth]{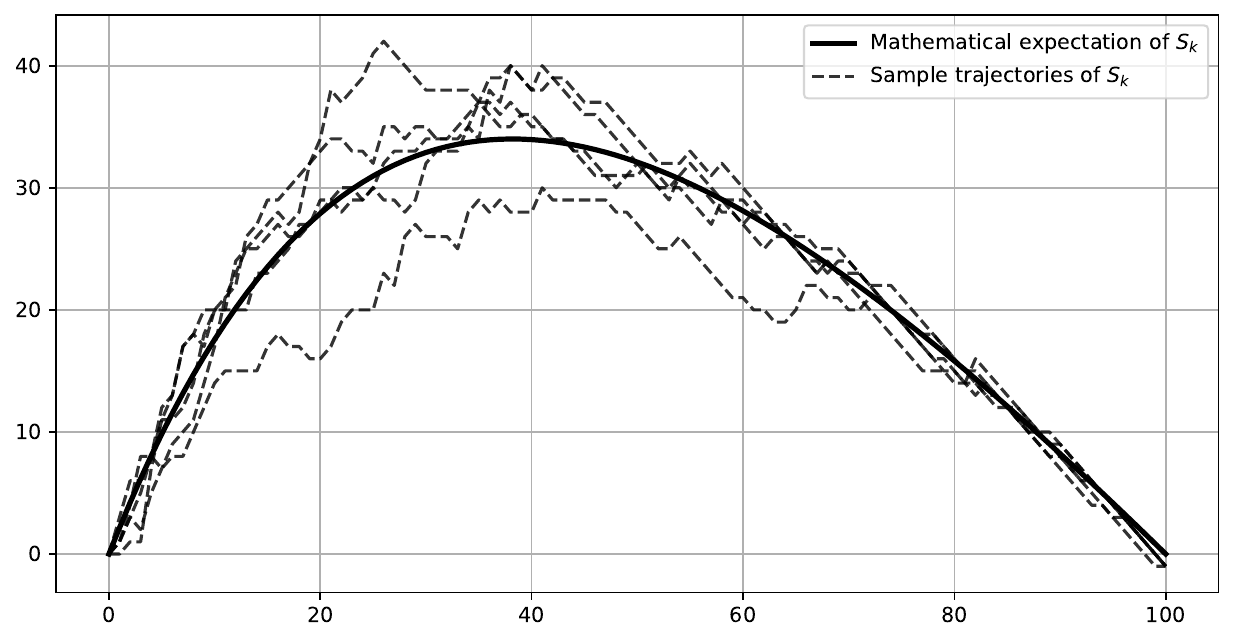} 
    \captionsetup{justification=centering}
    \caption{Graph of the mathematical expectation and sample realizations of \(S_k\) \\ for \(n=100\), \(p = 3/n\).}
    \label{ex1}
\end{figure}

In the theory of random walks, non-negative trajectories are commonly referred to as "meanders", trajectories that return to zero are called "bridges", and those that return to zero for the first time at the final step are known as "excursions". Hence, the problem of determining the graph's connectivity is reduced to studying the asymptotic behavior of the conditional probability that a bridge is an excursion.

\subsection{Main Result}\label{main theorem}
Using the representation obtained in Lemma \ref{connect}, we describe the asymptotics of the connectivity probability of the graph \(G(n,p)\) for various behaviors of the parameter \(p = p(n)\) as \(n \to \infty\). The results are summarized in the following theorem.

\begin{Thm}[On the Connectivity Probability of a Graph]\label{main res}
Let \(G(n,p)\) be a graph in the Erdős–Rényi model with edge probability \(p = c_n/n\), and let \(P_n(p)\) denote the probability that the graph \(G(n,p)\) is connected.
\begin{enumerate}[1)]
\item 
Suppose that \(c_n \to +\infty\) as \(n \to \infty\). Then
\begin{equation}\label{my_res1}
P_n(p) \sim  \left( 1-\left( 1-\frac{c_n }{n}\right)^{n}\right)^{n-1},\quad 
n\to \infty.
\end{equation}
\item
Suppose that \(c_n \to c \in (0,+\infty)\) as \(n \to \infty\). Then
\begin{equation}\label{Ptherem}
P_n(p) \sim (1-e^{-c})\left(1- \frac{c \ e^{-c}}{ 1-e^{-c}}\right)\left(1-\left( 1-\frac{c}{n}\right)^{n}\right)^{n-1}, \quad n \to \infty,
\end{equation}
\item  
Suppose that \(c_n=o(1)\) and, moreover, \(c_n n^{1/2}/\ln n \to+\infty\) as \(n\to\infty\). Then 
\begin{equation}\label{my_res3}
P_n(p) \sim \frac{1}{2} c_n^2 \left( 1-\left( 1-\frac{c_n }{n}\right)^{n}\right)^{n-1},\quad 
n\to \infty.
\end{equation}
\item 
Suppose that \(c_n=o(1/n)\). Then
\begin{equation}\label{my_res4}
P_n(p) \sim \frac{1}{n} \left( 1-\left( 1-\frac{c_n }{n}\right)^{n}\right)^{n-1} \sim \frac{c_n^{n-1}}{n} ,\quad 
n\to \infty.
\end{equation}
\end{enumerate}
\end{Thm}
\begin{Rem}
The results (\ref{my_res1}), (\ref{Ptherem}), and (\ref{my_res4}) correspond to the already known asymptotics (\ref{res1}), (\ref{res2}), and (\ref{res3}). However, the relation (\ref{my_res3}) appears to be a new result.
\end{Rem}

By virtue of Lemma \ref{connect}, the proof of Theorem \ref{main res} reduces to the analysis of the probability
\begin{equation*}
\mathbf{P}(S_k \geq 0, \ 0< k < n\mid S_{n} = -1 ) = 
\frac{\mathbf{P}(S_k \geq 0, \ 0< k < n, S_{n} = -1 )}{\mathbf{P}(S_{n} = -1 )},
\end{equation*}
where \(S_k = \sum_{i=1}^{k} X_i\), the \(X_i\) are independent with \(X_i + 1\sim Poiss\left(\lambda_{n,i}\right)\), and \(S_i + i \sim Poiss(\eta_{n,i})\),
 \begin{equation} 
 \label{lambda}
 \lambda_{n,i} = \frac{c_n }{b_n} \left( 1-\frac{c_n }{n}\right)^{i-1}, \quad b_n = 1-\left( 1-\frac{c_n }{n}\right)^{n},
 \end{equation}
\begin{equation*}
\eta_{n,i} = \sum_{j=1}^i \lambda_{n,j}  = \frac{1\ -\left( 1-c_n /n \right)^{i}}{b_n} n.
\end{equation*}

To proceed, we will require some results concerning inhomogeneous random walks.

\section{Auxiliary Results}\label{Results}

\subsection{Probability That a Homogeneous Walk Forms a Meander}
\label{positivity edges}
In this section we find the probabilities of positivity for some random walks with identically distributed steps.

\begin{Lemm}\label{lemma1}
Let $(Y_1,\dotsc, Y_i, \dotsc)$ be a sequence of independent identically distributed random variables, $Y_1 + 1 \sim Poiss(1)$, and let $k\in\{1,2,\dotsc, n\}$ be an arbitrary parameter. Define $S_n = (k-1) + \sum_{i=1}^n Y_i $. Then
\begin{equation}\label{lemmma1_result}
    \mathbf{P}\left(S_i \geq 0 , i<n, S_{n} = -1 \right) = \frac{k}{n}\mathbf{P}\left(S_{n} = -1 \right).
\end{equation} 
\end{Lemm}
\noindent
This result can be found in the book \cite{такач1971комбинаторные} on page 33.

\begin{Lemm}\label{left_lemma}
Let $(Y_1,\dotsc, Y_i, \dotsc)$ be a sequence of independent identically distributed random variables, $Y_i + 1 \sim Poiss(\gamma )$, $\gamma >1$, 
$S_n = \sum_{i=1}^n Y_i $.  
\begin{enumerate}[1)]
\item Then
\begin{equation}\label{left_lemmma_result}
    \mathbf{P}\left(S_k \geq 0 , k>0 \right) = 1 + \frac{1}{\gamma } W_0\left( -\frac{\gamma  }{e^{\gamma  }}\right), 
\end{equation} 
where $W_0(x)$ is the Lambert function, i.e. a function from $(-1/e,0)$ to $(-1,+\infty)$ such that for $x>-1$ the equality $W_0(xe^x)=x$ holds.
\\
\item Let $\gamma=\gamma _n=1+d_n$, $n^{-1/2} \ln n<d_n$, $d_n=o(1)$, 
$m_n$: $m_n d_n^2/\ln d_n\to-\infty$, as $n\to\infty$. Then
$$
\mathbf{P}(S_k\geq 0, 0<k<m_n)\sim 2 d_n,\ n\to\infty.
$$
\end{enumerate}
\end{Lemm}
\begin{proof}
1) Consider the sequence $\{q^{S_n}\}$ and find such a $q\in (0,1)$ for which this is a martingale. For this, the relation 
$$
\mathbb{E} q^{Y_i} = 1
$$
must hold. 
Since $Y_i \sim Poiss(\gamma ) - 1$, we have
\begin{equation}\label{martingal_equation}
\mathbb{E} q^{Y_i} = \exp(\gamma (q-1))/q = 1,
\end{equation}
hence,
$$
\exp(-\gamma  q)(-\gamma  q) = -\gamma   e^{-\gamma }.
$$
Thus, 
$$
q = -\frac{1}{\gamma } W_0\left( -\frac{\gamma  }{e^{\gamma  }}\right) < 1.
$$
Let $\tau = \inf \{t>0: S_t = -1\}$ denote the first time the random walk reaches $-1$. Now consider our martingale at time $\tau_n = \min(\tau, n)$. By the optional stopping theorem $\mathbb{E} q^{S_{\tau_n}} = 1$. On the other hand,
$$
\mathbb{E} q^{S_{\tau_n}}  = \frac{1}{q} \mathbf{P}\left( \tau \le n \right) + \sum_{k=0}^{\infty} q^k \  \mathbf{P}\left(S_i \geq 0 , i \le n, S_n = k \right).
$$
\newpage
\noindent Note that
\begin{multline*}
\sum_{k=0}^{\infty} q^k \ \mathbf{P}\left(S_i \ge 0 , i \le n, S_n = k \right) \ \le \ \sum_{k=0}^{\infty} q^k \  \mathbf{P}\left( S_n = k \right) \le\\ \le q^{n^{1/3}} + \mathbf{P}\left( S_n \le n^{1/3} \right) 
= o(1), \quad n \to \infty.
\end{multline*}
Then, $\mathbf{P}\left( \tau \le n \right) \to q$ as $n \to \infty$, and hence $\mathbf{P}\left(S_k \geq 0 ,  k>0 \right) = 1-q$.
\\

2) Introduce, as in the previous part, the martingale $q_n^{S_n}$, where we define $q_n\in (0,1)$ as the solution of the equation
$$
{\bf E}q_n^{Y_i} = 1,
$$
given by the relation
$$
\exp(\gamma_n (q_n - 1)) = q_n.
$$
We will show that $q_n$ admits the representation
$$
q_n = 1 - 2 d_n + O(d_n^2),\ n\to\infty.
$$
Indeed, $q_n\to 1^-$, as $n\to+\infty$, since any limit point $z$ of the bounded sequence $\{q_n-1\}$ satisfies the equation $e^{z}=1+z$, which has no nonzero solutions due to the strict convexity of the exponential function. Consequently,
$$
\exp(\gamma_n(q_n -1)) = 1 + \gamma_n (q_n - 1) + \frac{1}{2} \gamma_n^2 (q_n-1)^2 + O\left(\left(q_n-1\right)^3\right),\ n\to\infty,
$$
hence
\begin{equation}
\label{ExpansionWLambert}
(q_n - 1)\left(d_n + \frac{\gamma_n^2 (q_n-1)}{2}+O\left(\left(q_n-1\right)^2\right)\right) = 0,\ n\to\infty. 
\end{equation}
Thus, $q_n = 1 + 2d_n + \varepsilon_n d_n$, as $n\to\infty$, where $\varepsilon_n\to 0$, as $n\to\infty$. Substituting this expression into the relation~(\ref{ExpansionWLambert}), we obtain $\varepsilon_n = O(d_n)$, as $n\to\infty$.

As before, by the optional stopping theorem applied at the stopping time $\min(\tau, m_n)$, we have
$$
1 = q_n^{-1} \mathbf{P}(\tau<m_n)+ \sum_{k=0}^{\infty} q_n^k \ \mathbf{P}(S_i\ge 0,i\le m_n, S_{m_n}=k),
$$
from which we get
$$
\mathbf{P}(S_k\geq 0, 0<k<m_n) = 1- q_n + q_n\sum_{k=0}^{\infty} q_n^k \ \mathbf{P}(S_i\ge 0,i\le m_n, S_{m_n}=k).
$$
Also, for $a_n=d_n m_n/2$, we have the inequalities
$$
\sum_{k=0}^{\infty} q_n^k \ \mathbf{P}(S_i\ge 0,i\le m_n, S_{m_n} =k)\le 
\mathbf{P}(S_{m_n} \le a_n) + q_n^{a_n}.
$$
Moreover, as $n\to\infty$ 
$$
q_n^{a_n}=\left(1-2d_n+O(d_n^2)\right)^{d_n m_n/2}=
e^{-d_n^2 m_n (1+o(1))}.
$$
\newpage
Since the condition $-d_n^2 m_n\le 2 \ln d_n$ holds for all sufficiently large $n$, it follows that $q_n^{a_n}$ is $o(d_n)$ as $n\to\infty$.
Also, by Markov's inequality for any positive $h$ the following estimates hold
\begin{gather*} 
\mathbf{P}(S_{m_n} \le a_n) = \mathbf{P}(-S_{m_n} \ge -a_n) \le \\
\le e^{h(m_n+a_n)} {\bf E}e^{-h(S_{m_n}+m_n)}=
e^{h(m_n+a_n)+\gamma_n m_n(e^{-h}-1)}. \nonumber
\end{gather*}
Moreover,
$$
h(m_n+a_n)+\gamma_n m_n(e^{-h}-1) \le h^2 \gamma_n m_n/2 - 
h m_n d_n/2.
$$
Choosing $h=d_n/2$, we obtain
$$
\mathbf{P}(S_{m_n} \le a_n)\le e^{-m_n d_n^2/8+d_n^3 m_n/8}\le e^{2\ln d_n}
$$ 
for all sufficiently large $n$, where the right-hand side is $o(d_n)$ as $n\to\infty$.
Thus,
$$
\mathbf{P}(S_k\geq 0, 0<k<m_n) = 1 - q_n + o(d_n) = 2 d_n + o(d_n),\ n\to\infty.
$$
This completes the proof of Lemma \ref{left_lemma}.
\end{proof}

Note that the formula (\ref{left_lemmma_result}) from Lemma \ref{left_lemma} simplifies if $\gamma $ has the form given below.
\begin{Co}\label{left co}
If $\gamma  = \lambda /\left( 1-e^{-\lambda }\right)$, 
$$
\mathbf{P}\left(S_k \geq 0 ,  k>0 \right) = 1-e^{-\lambda}.
$$
\begin{proof}
We verify the condition (\ref{martingal_equation}) for $q = e^{-\lambda}$, 
$$ \exp\left(\gamma (q-1)\right)/q =  \exp\left((\lambda/(1-e^{-\lambda }))(e^{-\lambda}-1)\right)/e^{-\lambda} =  \exp(-\lambda) / e^{-\lambda} = 1, $$
which is what needed to be proved.
\end{proof}
\end{Co}

\begin{Lemm}\label{right_lemma} 
Let $(Y_1,\dotsc, Y_i, \dotsc)$ be a sequence of independent identically distributed random variables, $1 - Y_i \sim Poiss(\gamma), \gamma < 1$, $S_n = \sum_{i=1}^n Y_i$.  
\begin{enumerate}[1)]
\item Then
\begin{equation*}
    \mathbf{P}\left(S_k > 0 , k > 0 \right) = 1-\gamma .
\end{equation*} 
\item For $\gamma =1-d_n$, $n^{-1/2} \ln n < d_n$, $d_n=o(1)$, as $n\to\infty$, the following relation holds
$$
\mathbf{P}(S_k > 0, k\le m_n)\sim d_n ,\ n\to\infty,
$$
where $m_n: m_nd_n^2/\ln d_n\to -\infty$.
\end{enumerate}
\end{Lemm}
\begin{proof}
1) Let us find the probability of strict positivity of the walk 
$$P_0 = \mathbf{P}\left(S_k > 0 , k>0 \right).$$
Note that the number of returns to $0$ (denote it by $N_0$) has a geometric distribution. We can determine its parameter by computing its expectation:
$$ \mathbb{E} N_0 = 
\sum_{k=1}^{+\infty} \mathbf{P}(S_k = 0) = \sum_{k=1}^{+\infty} e^{-\gamma  k} \ ( k\gamma  )^{k} /k!\ =\sum_{k=1}^{+\infty}\left( e^{-\gamma  } k\gamma  \right)^{k} /k! \ .
$$
The evaluation of this sum is given in \cite{такач1971комбинаторные} on page 78 as an exercise. For completeness, we solve this exercise and show that $ \mathbb{E} N_0 = \gamma / (1-\gamma)$. Note that
\begin{gather*}
\sum _{k=1}^{\infty }\frac{\gamma  ^{k} k^{k}}{k!} e^{-\gamma  k} =\sum _{k=1}^{\infty }\sum _{i=0}^{\infty }\frac{\gamma  ^{k} \ k^{k}}{k!\ } \ \frac{( -\gamma )^{i} \ k^{i}}{i!} =
\\ = \sum _{j=1}^{\infty }\frac{\gamma ^{j}}{j!} \ \sum _{k=1}^{j}\frac{j!\ k^{j} \ ( -1)^{j-k}}{k!\ ( j-k) !}  = \sum _{j=1}^{\infty }\frac{\gamma ^{j}}{j!} \ \sum _{k=1}^{j} \binom{j}{k}  k^{j}  ( -1)^{j-k} .
\end{gather*}
Using the identity
$
\sum _{k=1}^{j} \binom{j}{k} k^{j}  ( -1)^{j-k} = j!,
$
we obtain $\mathbb{E} N_0 = \gamma/(1-\gamma )$. Hence, the parameter of our geometric distribution is $\gamma$, therefore
$$
\mathbf{P}\left(S_k > 0 , \forall k \right) = 1 - \gamma.
$$
Thus, $P_0 = 1-\gamma$.
\\

2) Note that by part 1
$$
\mathbf{P}(S_k\ge 0,k\le m_n) = 
d_n - \mathbf{P}(\exists i>m_n: S_i<0).
$$
However, 
$$
\mathbf{P}(\exists i>m_n: S_i<0) \le 
\sum_{i=m_n}^{\infty} \mathbf{P}(Z_i\ge i),
$$
where $Z_i\sim Poiss(i(1-d_n))$. By Markov's inequality for any positive $h$ the following holds
$$
\mathbf{P}(Z_i\ge i)\le e^{-hi} {\bf E}e^{hZ_i} = 
e^{i(1-d_n)(e^h-1) - hi}.
$$
For $h=d_n/2$ we obtain
$$
(1-d_n)(e^h-1) - h = 
-d_n^2/2 + d_n^2/8+O(d_n^3) = 
-3 d_n^2/8+ O(d_n^3),
$$
For sufficiently large $n$, the right-hand side is bounded above by $-d_n^2/4$, hence 
$$
\sum_{i=m_n}^{\infty} \mathbf{P}(Z_i\ge i)\le 
\frac{e^{-d_n^2 m_n/4}}{1-e^{-d_n^2/4}} \le 32\ d_n^{-2} e^{4\ln d_n}.
$$
The right-hand side of the above expression is $o(d_n)$ as $n\to\infty$. Thus, the lemma is proved.
\end{proof}

\subsection{Lemma on the Comparison of Poisson Bridges}\label{Approximation lemm}
\begin{Lemm}
\label{StocDomLemm}
Let $X_{i,l}\sim Poiss(\gamma_{i,l})$, $i\le n$, $l=1,2$, be independent, where 
\begin{equation}
\label{GammaCond}
\frac{\sum_{i=1}^{j} \gamma_{i,1}}{\sum_{i=1}^{n}\gamma_{i,1}}
\ge \frac{\sum_{i=1}^{j} \gamma_{i,2}}{\sum_{i=1}^{n}\gamma_{i,2}},\ j\le n.
\end{equation}
The variables $X_{i,l}$ define the random walks
$$
S_{j,1} = \sum_{i=1}^{j} X_{i,1}, \ \ S_{j,2} = \sum_{i=1}^{j} X_{i,2}.
$$
Then, for any $x_j$, $j\le n$, and any $y$ from $\mathbb{N}\cup\{0\}$, the following inequality holds
\begin{equation}\label{GammaCond2}
\mathbf{P}\left(\left.S_{j,1}\ge x_j, \ j\le n\right|S_{n,1} = y\right)\ge 
\mathbf{P}\left(\left. S_{j,2}\ge x_j, \ j\le n\right| S_{n,2} = y\right).
\end{equation}

\end{Lemm}
\begin{proof}
Consider the Poisson processes $N_t^l$ with intensities $\lambda_l$,
$$
\lambda_l = \sum_{i=1}^{n} \gamma_{i,l}, \quad
t_{j,l} = \frac{\sum_{i=1}^{j} \gamma_{i,l}}{\sum_{i=1}^{n} \gamma_{i,l}},\quad j\le n,\ l=1,2.
$$
Then 
$$
\left(S_{j,l}, \ j\le n\right)\stackrel{d}{=} (N_{t_{1,l}}^l,N_{t_{2,l}}^l,\dotsc, N_{t_{n,l}}^l),\quad l\in \{1,2\}.
$$
Hence, 
\begin{equation*}
\begin{gathered}
\mathbf{P}\left(\left.S_{j,1}\ge x_j, \ j\le n\right|S_{n,1} = y\right) =
\mathbf{P}\left(\left.N_{t_{j,1}}^1\ge x_j, \ j\le n\right|
N_{1}^1=y\right)=\\
= \mathbf{P}\left( \left. \tau_{x_j}^1<t_{j,1},\ j\le n \right| N_1^1 = y \right)= 
\mathbf{P}(R_{x_j}<t_{j,1},\ j\le n),
\end{gathered}
\end{equation*}
where $\tau_i^l$ are the points (jump times) of the Poisson process corresponding to $(N_t^l)$, and $R_i$, $i\le y$, denote the order statistics of $y$ independent $U[0,1]$ random variables. In the last equality we used the conditional property of the Poisson process ~\cite{grimmett2020probability}. The inequality
$$
\mathbf{P}(R_{x_j}<t_{j,1},\ j\le n)\ge 
\mathbf{P}(R_{x_j}<t_{j,2},\ j\le n)
$$
follows immediately from the definition of $t_{j,l}$ and condition~(\ref{GammaCond}). This completes the proof of Lemma \ref{StocDomLemm}.
\end{proof}

\begin{Rem}\label{StocDomRem}
    Note that equality in expression (\ref{GammaCond}) implies equality of the corresponding conditional probabilities (\ref{GammaCond2}).
\end{Rem}
\subsection{Inequality for the Probability of the Inhomogeneous Random Walk Hitting -1}\label{positivity mid}

In this section we prove a lemma that allows us to estimate the probability that the process reaches $-1$ at a step far from both the beginning and the end.

\begin{Lemm}\label{mid_lemma}
Let $n\ge 3$, and let $S_k = \sum_{i=1}^{k} X_i$, where $X_i$ are independent random variables and $X_i+1\sim Poiss\left(\lambda_{n,i}\right)$, where $\lambda_{n,i}$ are given by the relation~(\ref{lambda}).
\begin{enumerate}[1)]
    \item Then for any natural $m<n/2$ and $1\leq c_n<n$, the following inequality holds
\begin{equation*}
\mathbf{P}\left(\exists i\in [m,n-m]: S_i=-1 \mid S_{n} = -1\right) \le 400 \cdot 0.99^{m}.
\end{equation*}
\item
If $c_n < 1$, then the following inequality holds
\begin{equation*}
\mathbf{P}\left(\exists i\in [m,n-m]: S_i=-1 \mid S_{n}= -1\right) \le \frac{500}{c_n^2 \sqrt{ m}} \exp\left(-\frac{m c_n^2}{200}\right).
\end{equation*}

\end{enumerate}
\end{Lemm}
\begin{proof}
We use the inequalities
$$
\sqrt{2\pi i}(i/e)^i \le i!\le 2\sqrt{2\pi i}(i/e)^i,
$$
which hold for all $i\ge 1$. 
Note that
\begin{equation}\label{P_S_n=-1}
\mathbf{P}(S_{n}=-1) = \frac{n^{n-1} e^{-n}}{(n-1)!} =\frac{n^{n} e^{-n}}{n!}\ge \frac{1}{2\sqrt{2\pi n}}.
\end{equation}
Consider the probability
$$\mathbf{P}\left(\exists i\in [m,n-m]: S_i=-1 , \ S_{n}=-1\right).$$
We bound it from above by the sum
\begin{equation}\label{mid}
\sum_{i=m}^{n-m} \mathbf{P}(S_{i}=-1) \mathbf{P}(S_{n}-S_i=0) .
\end{equation}
We use the relations 
$
S_{i} + i \sim Poiss\left(\eta_{n,i}\right),\quad S_{n} - S_{i} + (n-i) \sim  Poiss\left(n-\eta_{n,i}\right).
$
Then for $\displaystyle m \le i\le n/2$ we obtain
\begin{multline}\label{mid_left}
  \mathbf{P}( S_{i} = -1) = \exp(-\eta_{n,i})\frac{( \eta_{n,i})^{i-1}}{(i-1)!}
  \le  
 \exp(-\eta_{n,i})\frac{( \eta_{n,i})^{i}}{i!}
 \le
 \\
\le  
 \frac{( \eta _{n,i} /i)^{i} }{\sqrt{2\pi i}\exp( \eta_{n,i}-i)}
 \le \frac{1}{\sqrt{2\pi i}}\left(\frac{\eta_{n,i} /i}{\exp( \eta_{n,i} /i  - 1) }\right)^{i} \le \ \frac{h_1(c_n)^i}{\sqrt{2\pi i}},
\end{multline}
where 
$$h_1(c_n):=\max_{m\le i\le n/2} \left(\frac{\eta_{n,i} /i }{\exp( \eta_{n,i} /i - 1) }\right) = \max_{m\le i\le n/2} \exp(\psi(\eta_{n,i}/i)),$$
with $\psi(x) = \ln x + 1 - x$, $x>0$. The same estimates yield the inequality 
\begin{equation}
\label{IneqBadi}
\mathbf{P}(S_i=-1)\le \frac{1}{\sqrt{2\pi i}}
\end{equation}
for an arbitrary $i$. Note that $\psi(x)$ attains its maximum, equal to zero, at $x=1$, decreases on the interval $(1,+\infty)$ and increases on $(0,1)$. Since the value $\eta_{n,i} /i$ is the arithmetic mean of $\lambda_{n,j}$, $j=1, \dotsc,i$, and the $\lambda_{n,j}$ are monotonically decreasing, the minimum of $\eta_{n,i} /i$ for $i \le n/2$ is attained at $i = n/2$. Consequently,
$$
\underset{m\le i\le n/2 }{\min} \ \eta _{n,i} /i\ =
2\frac{1-\left( 1-c_n/n\right)^{n/2}}{1-\left( 1-c_n/n\right)^{n}}= 
\frac{2}{1+\left( 1-c_n/n\right)^{n/2}}>1,
$$
from which it follows that
\begin{equation}
\label{left_average}
\max_{m\le i\le n/2} \psi\left(\eta_{n,i}/i\right)\le \psi\left(\frac{2}{1+\left( 1-c_n/n\right)^{n/2}}\right).
\end{equation}
For $c_n\ge 1$ the right-hand side of (\ref{left_average}) is bounded by
$$
\psi\left(\frac{2}{1+\left( 1-c_n/n\right)^{n/2}}\right)\le \psi\left(\frac{2}{1+e^{-1/2}}\right) \leq \psi(6/5)=\ln (6/5)-1/5,
$$
where in the first step we used the inequalities
$$
\left(1-\frac{c_n}{n}\right)^{n}\le \left(1-\frac{1}{n}\right)^{n} = \left(1+\frac{1}{n-1}\right)^{-n} \le e^{-1}.
$$
In the last inequality we used the monotonicity of the sequence $(1+(n-1)^{-1})^n$, as proved, for example, in Example 13 of Section 1, Chapter III of the book~\cite{зорич1997математический}.
For $c_n<1$, the right-hand side of (\ref{left_average}) is bounded by
\begin{multline}
\psi\left(\frac{1}{1-c_n/4 + c_n^2/16}\right)\le - \frac{1}{4}\left(\frac{1}{1-c_n/4 + c_n^2/16}-1\right)^2 = \\ 
= - \frac{1}{4}\left(\frac{c_n/4 - c_n^2/16}{1-c_n/4 + c_n^2/16} \right)^2 \le
-\frac{1}{4}\left(\frac{3c_n}{16}\right)^2
\le -\frac{c_n^2}{200},
\end{multline}
where we used the inequalities
$$
(1-x)^j\le 1 - xj + j^2 x^2/2,\quad 
\psi(1+x)\le -\frac{x^2}{4},
$$
which hold for all $x\in [0,1]$ and $j\ge 2$. Hence,
\begin{eqnarray}
h_1(c_n)\le \exp(\ln(6/5)-1/5) \leq 0.99,\quad c_n\ge 1, \label{mid_h1}\\ 
h_1(c_n)\le \exp(-c_n^2/200),\quad c_n\le 1. \label{mid_h2}
\end{eqnarray}
Similarly, for $n/2 \le i\le n -m$:
\begin{equation}\label{mid_right}
\mathbf{P}( S_{n} - S_{i} = 0) \leq \frac{1}{\sqrt{2\pi (n-i)} }\exp\left((n-i) \psi\left(\frac{n-\eta_{n,i}}{n-i}\right)\right)\le \frac{h_{2}(c_n)^{n-i}}{\sqrt{2\pi (n-i)}},
\end{equation}
where 
$$
h_{2}(c_n) =\max_{n/2 \le i\le n -m} \exp\left(\psi\left(\frac{n-\eta_{n,i}}{n-i}\right)\right).
$$
Moreover, as before,
\begin{equation}
\label{IneqBadi2}
\mathbf{P}(S_{n}-S_i=0)\le \frac{1}{\sqrt{2\pi (n-i)}}
\end{equation}
for all $i\in (n/2,n-m)$. 
Note that 
$$
\max_{n/2 \le i\le n-m} \frac{n-\eta_{n,i}}{n-i} \le 2 - 2\frac{1-\left( 1-c_n/n\right)^{n/2}}{1-\left( 1-c_n/n\right)^{n}} = 2- 2\frac{1}{1+\left( 1-c_n/n\right)^{n/2}} .
$$
The same estimates as before show that
$$
h_2(c_n)\le
\exp(\psi(4/5)) = \exp(\ln (4/5) + 1/5)\le 0.99
$$
for $c_n\ge 1$, and for $c_n<1$
$$
h_2(c_n)\le \exp\left(-\frac14\left(1-\frac{1}{1-c_n/4 + c_n^2/16}\right)^2\right)\le -\frac{c_n^2}{200}.
$$
Therefore, applying to (\ref{mid}) for $m \le i \le n/2$ the estimates (\ref{mid_left}) and (\ref{IneqBadi2}), and for $n/2 < i \leq n-m$ the inequalities (\ref{mid_right}) and (\ref{IneqBadi}), 
we obtain the inequality
\begin{multline}\label{mid_lem_sum}
\mathbf{P}(\exists i\in [m,n-m]: S_i=-1\mid S_{n}=-1)\le \\ \le 2\sum_{i=m}^{n/2}\frac{2 \sqrt{n}}{\sqrt{2\pi i(n-i)}} h_1(c_n)^i
\le \frac{4}{\sqrt{\pi m}}\sum_{i=m}^{n/2} h_1(c_n)^i.
\end{multline}
Thus, using (\ref{mid_h1}), for $c_n\ge 1$ we obtain the inequality
\begin{equation*}
\mathbf{P}\left(\exists i\in [m,n-m]: S_i= -1 \mid S_{n}= -1\right)
\le 4\sum_{i=m}^{n/2} (0.99)^i \le 400\cdot 0.99^m.
\end{equation*}
For $c_n<1$, using (\ref{mid_h2}), the right-hand side of (\ref{mid_lem_sum}) is bounded by
$$
\frac{4\exp\left(-m c_n^2/200\right)}{\sqrt{2\pi m}(1-\exp\left(-c_n^2/200\right))}\le \frac{500}{\sqrt{m} c_n^2}\exp\left(-\frac{m c_n^2}{200}\right).
$$
Thus, Lemma \ref{mid_lemma} is proved.
\end{proof}

\section{Proof of the Theorem}\label{main proof}
\begin{proof}[Proof of Theorem \ref{main res}]
From Lemma \ref{connect} we know that
\begin{equation}\label{lem_result_theorem}
P_n(p) = \left( 1-\left( 1-\frac{c_n}{n}\right)^{n}\right)^{n-1} \mathbf{P}(S_k \geq 0, \ 0< k < n \mid S_{n} = -1 ).
\end{equation}
We also know that 
\begin{equation}\label{P_S_n sim}
\mathbf{P}(S_{n}=-1) \sim \frac{1}{\sqrt{2\pi n}},\quad n\to\infty.
\end{equation}
We need to find the asymptotic behavior of
$$
P_n:= \mathbf{P}(S_k\ge 0,\ 0<k<n, \ S_{n}=-1).
$$ 
We will prove parts 2 and 3 of the theorem by considering, for a properly chosen sequence $\{m_n,\ n\ge 1\}$, the random walk on three intervals 
$$
I_1 = [1,\dotsc,m_n], \quad I_2 = (m_n,\dotsc,n-m_n), \quad I_3 = [n-m_n,\dotsc,n-1].
$$
Then, using the obtained results, we will prove parts 1 and 4.

\subsection*{Proof of Case 2}
    
Consider the case $c_n \to c$. Set $m_n = n^{1/5}$. By virtue of Lemma~\ref{mid_lemma} 
\begin{multline*}
\mathbf{P}(S_k \ge 0, \ k\in I_1 \cup I_3;\  \exists \ l \in I_2 : S_l = -1 , \ S_{n} = -1 ) \le
\\
\le \ \mathbf{P}(\exists \ l \in I_2 : S_l = -1 , \ S_{n} = -1 )  = o\left(n^{-2}\right),\quad n \to \infty.
\end{multline*}
Thus, as $n \to \infty$
\begin{equation}\label{Ptilde}
\widetilde{P}_n := \mathbf{P}(S_k \ge 0, \ k\in I_1 \cup I_3, \ S_{n} = -1 ) = P_n + o\left(n^{-2}\right).
\end{equation}
Now we introduce the following notations for the probabilities
\begin{align*}
&P_{l}^{I_1} = \mathbf{P}(S_k \ge 0, \ k\in I_1, \ S_{m_n} = l),\\[1mm]
&P_{l,r}^{I_2} = \mathbf{P}(S_{n-m_n} = r \mid S_{m_n} = l),\\[1mm]
&P_{r}^{I_3} = \mathbf{P}(S_k \ge 0, \ k\in I_3, \ S_{n} = -1 \mid S_{n-m_n} = r).
\end{align*}
By the law of total probability we have
\begin{equation*}
    \widetilde{P}_n =
    \mathbf{P}(S_k \ge 0, \ k\in I_1 \cup I_3, \ S_{n} = -1 ) 
    =\sum _{l\geq 0}\sum_{r\geq0} P_{l}^{I_1}\, P_{l,r}^{I_2}\, P_{r}^{I_3}.
\end{equation*}
Moreover, $S_{n-m_n} - S_{m_n} + n -2m_n \sim Poiss(\mu_n)$, where 
$
\mu_n = \eta_{n,n-m_n}  -\eta_{n, m_n},
$ and
\begin{eqnarray}
m_n \leq \eta_{n, m_n} \leq \lambda_{n,1}m_n = O(m_n) = O(n^{1/5}),\quad n \to \infty, \label{mu_n_1}
\\  
n \geq \mu_n \geq n -m_n- \lambda_{n,1}m_n = O(n),\quad n \to \infty. \label{mu_n_2}
\end{eqnarray}
Since the maximum of $e^{-\mu} \mu^{k}/k!$ is attained at $k = \lfloor\mu\rfloor$ and using the inequality $ i!\geq \sqrt{i}(i/e)^i$, for any $l, r$ we have the estimate
\begin{equation}
\label{mu_n_3}
\begin{gathered}
P_{l,r}^{I_2} = \frac{e^{-\mu_n} \mu_n^{n-2m_n-( l-r)}}{( n-2m_n-( l-r))!}\le \frac{e^{-\mu_n} \mu_n^{\lfloor\mu_{n}\rfloor}}{\lfloor\mu_{n}\rfloor!}
\leq \\
\leq \frac{e^{-\mu_{n}} \mu_{n}^{\lfloor \mu_{n} \rfloor } }{\sqrt{\lfloor \mu_{n} \rfloor}\, e^{-\lfloor\mu_n \rfloor} \lfloor \mu_{n} \rfloor^{\lfloor \mu_{n} \rfloor }} 
\leq 
\frac{1}{\sqrt{\lfloor \mu_{n} \rfloor}},
\end{gathered}
\end{equation}
where in the last inequality we used the fact that the function $e^{-x} x^{\lfloor \mu_{n} \rfloor}$ attains its maximum at $x=\lfloor \mu_{n} \rfloor$.

Now, consider the value of $P_{l,r}^{I_2}$ for $l, r \le n^{2/5}$. Denote $\displaystyle n - \mu_n$ by $\displaystyle a$ and $\displaystyle 2m_n+( l-r)$ by $\displaystyle b$. Then, as $n\to\infty$, the quantities $a,b$ are of order $O(n^{2/5})$. Consequently, we obtain for $n\to\infty$
 \begin{multline}\label{Stirling proof}
\exp( -(n-a))\frac{(n-a)^{(n-b)}}{(n-b)!} = \frac{\exp(-(b-a))}{\sqrt{2\pi (n-b)}(1+o(1))}  \left(\frac{n-a}{n-b}\right)^{n-b} =
\\
=\frac{(1+o(1))}{\sqrt{2\pi n}} \exp(a-b) \exp \left((n-b) \ln\left( 1+\frac{b-a}{n-b}\right)\right) = 
\\
=\frac{1+o(1)}{\sqrt{2\pi n}} \exp(a-b) \exp \left((b-a) + O\left(\frac{(b-a)^{2}}{n-b}\right)\right) = \frac{1+o(1)}{\sqrt{2\pi n}}.
\end{multline}
Then one can assert that
\begin{equation}\label{decomposition}
    \sum _{l=0}^{n^{2/5}}\sum _{r=0}^{n^{2/5}} P_{l}^{I_1}\, P_{l,r}^{I_2}\, P_{r}^{I_3} =
    \frac{1+o(1)}{\sqrt{2\pi n}} \left(\sum _{l=0}^{n^{2/5}} P_{l}^{I_1} \right) \left(\sum_{r=0}^{n^{2/5}}  P_{r}^{I_3}\right).
\end{equation}
Using (\ref{mu_n_3}) to bound the maximum of $P_{l,r}^{I_2}$ and the estimates (\ref{mu_n_1}), (\ref{mu_n_2}), we obtain
\begin{multline}\label{p123}
\widetilde{P}_n - \sum _{l=0}^{n^{2/5}} \sum _{r=0}^{n^{2/5}} P_{l}^{I_1}\, P_{l,r}^{I_2}\, P_{r}^{I_3} \le \sum _{l=n^{2/5}}^{\infty} P_{l}^{I_1} \cdot \underset{l,r}{\max}\, P_{l,r}^{I_2} =
\\
=\mathbf{P}(S_{m_n} > n^{2/5}) \cdot \underset{l,r}{\max}\, P_{l,r}^{I_2}
\le \frac{\eta_{n,m_n} - m_n}{n^{2/5}} \lfloor \mu_{n} \rfloor^{-1/2} =  o(n^{-1/2}), \quad n\to \infty,
\end{multline}
where in the last inequality we used Markov's inequality. 
Now, apply Lemmas \ref{left_lemma} and \ref{right_lemma} to determine the probabilities
$$
Q_1 = \sum _{l=0}^{n^{2/5}} P_{l}^{I_1},\quad Q_3 = \sum _{r=0}^{n^{2/5}}  P_{r}^{I_3}.
$$

Let us start by determining $Q_1$.
We cannot directly apply Lemma \ref{left_lemma} to $S_k = \sum_{i=1}^{k} X_i$, where the $X_i$ are independent random variables such that $X_i + 1 \sim Poiss\left(\lambda_{n,i}\right)$, since the random walk is inhomogeneous. Consider the sequence $Y_i = X_i + Z_i$, where $Z_i \sim Poiss\left(\lambda_{n,1}-\lambda_{n,i}\right)$, and apply the lemma to $\widetilde{S}_k = \sum_{i=1}^{k} Y_i$. Moreover, when $c_n \to c$
\begin{equation}\label{cn_to_c1}
\lambda_{n,1} \ =\ \frac{c_n}{1-\left( 1-c_n/n\right)^{n}} = \frac{c}{1-e^{-c}} + o(1), \quad n \to \infty.
\end{equation}
Therefore, by Corollary \ref{left co} we have
\begin{equation}\label{left couple1}
\sum _{l=0}^{n^{2/5}} \mathbf{P}\left(\widetilde{S}_k \geq 0,\ k\in I_1,\  \widetilde{S}_{m_n} = l \right) = 1-e^{-c} + o(1), \quad n \to \infty.
\end{equation}
Furthermore,
$$
0\le \mathbf{P}(\widetilde{S}_k \geq 0,\ k\in I_1 )- \mathbf{P}(S_k \geq 0,\ k\in I_1 ) \le \mathbf{P}\left(\sum_{i=1}^{m_n} Z_i \neq 0\right),
$$
where $\sum_{i=1}^{m_n} Z_i \sim Poiss(m_n\lambda_{n,1} - \eta_{n,m_n})$, and 
\begin{equation*}
 m_n\lambda_{n,1} - \eta_{n,m_n} \le m_n( \lambda_{n,1} -\lambda_{n,m_n}) = m_n\frac{c_n \left( 1-\left( 1-c_n/n\right)^{m_n-1}\right)}{1-\left( 1-c_n/n\right)^{n}} \le C\,\frac{m_n^{2}}{n},
\end{equation*}
which tends to zero as $n\to\infty$.
Therefore,
\begin{equation}\label{left couple2}
\mathbf{P}(\widetilde{S}_k = S_k,\ \forall k\in I_1) = 1 + o(1), \quad n \to \infty.
\end{equation}
Thus, from (\ref{left couple1}) and (\ref{left couple2}) it follows that
\begin{equation}\label{left couple3}
Q_1 = \sum _{l=0}^{n^{2/5}} P_{l}^{I_1} = 1-e^{-c} + o(1), \quad n \to \infty.
\end{equation}

Now, let us determine $Q_3$. First, consider the dual random walk on the last interval by reversing the order of the steps and changing their sign, i.e. $\widetilde{X}_i = -X_{n-i+1}$. Then the desired probabilities can be expressed in terms of the random walk $\widetilde{S}_k = \sum_{i=1}^k  \widetilde{X}_i$:
\begin{multline*}
\mathbf{P}(S_k \geq 0, \ k\in I_3, \ S_{n} = -1 \mid S_{n-m_n} = r) =
\\
= \mathbf{P}\left(\sum_{i = 0}^{k}X_{n-i} < 0,\ k < m_n,\ \sum_{i = 0}^{m_n-1}X_{n-i} = -r-1\right) =
\\
=\mathbf{P}(\widetilde{S}_k > 0,\ 0<k\le m_n,\ \widetilde{S}_{m_n} = r+1). 
\end{multline*} 
Therefore, we can write
$$
Q_3 = \sum _{r=0}^{n^{2/5}}  P_{r}^{I_3} = \sum _{r=0}^{m_n} P_{r}^{I_3} = \mathbf{P}(\widetilde{S}_k > 0,\ 0<k \le m_n).
$$
Similarly to the previous reasoning, we can apply Lemma \ref{right_lemma} to the inhomogeneous random variables $Y_i$ by introducing (possibly on an extended probability space) independent collections $Z_i \sim Poiss(\lambda_{n,n-i} -\lambda_{n,n})$, for $i\le n$, and $1-Y_i \sim Poiss(\lambda_{n,n})$, for $i\le n$, such that 
$$
Y_i+Z_i = \widetilde{X}_i. 
$$
Then,
\begin{gather*}
\sum_{i=1}^{m_n} Z_i \sim Poiss(\eta_{n,n} - \eta_{n,n-m_n} - m_n\lambda_{n,n}) ,
\\
\eta_{n,n} - \eta_{n,n-m_n} - m_n\lambda_{n,n} < m_n (\lambda_{n,n-m_n} - \lambda_{n,n}) = \\
= m_n \left(1-\left( 1-c_n/n\right)^{m_n} \right)  \frac{c_n \left( 1-c_n/n\right)^{n-m_n}}{1-\left( 1-c_n/n\right)^{n}} \le C\,\frac{m_n^{2}}{n},
\end{gather*}
so that
$$
\mathbf{P}(\widetilde{X}_k = Y_k,\ k \le m_n) = \mathbf{P}\left(\sum_{i=1}^{m_n} Z_i = 0\right) = 1 + o(1), \quad n \to \infty.
$$
Thus,
\begin{equation}\label{right couple3}
Q_3 = \sum _{r=0}^{n^{2/5}}  P_{r}^{I_3} = (1-\lambda_{n,n}) + o(1), \quad n \to \infty. 
\end{equation}
Since for $c_n \to c$, $n\to\infty$
\begin{equation*}
    \lambda_{n,n} = \frac{c_n \left( 1-c_n/n\right)^{n-1}}{1-\left( 1-c_n/n\right)^{n}} \sim \frac{c \, e^{-c}}{1-e^{-c}}, \quad n\to\infty,
\end{equation*}
substituting (\ref{left couple3}) and (\ref{right couple3}) into (\ref{decomposition}) and using (\ref{p123}), we obtain
\begin{equation}\label{Pn_tilda_final}
\widetilde{P}_n = \frac{(1+o(1))}{\sqrt{2\pi n}} (1-e^{-c})\left(1- \frac{c \, e^{-c}}{1-e^{-c}}\right), \quad n \to \infty.
\end{equation}
Substituting (\ref{Pn_tilda_final}) into (\ref{Ptilde}) and taking into account (\ref{lem_result_theorem}) and (\ref{P_S_n sim}), we obtain the desired result:
\begin{equation*}
P_n(p) \sim (1-e^{-c})\left(1- \frac{c \, e^{-c}}{1-e^{-c}}\right)
\left( 1-\left( 1-\frac{c}{n}\right)^{n}\right)^{n-1}, \quad n \to \infty.
\end{equation*}

\subsection*{Proof of Case 3}
Consider the case when $c_n\to 0$, $ c_n n^{1/2}/\ln n \to +\infty$, $n\to \infty$.  
Fix $m_n$ such that 
$$r_n = m_n c_n^2/|\ln c_n|\to +\infty, \quad m_n c_n = o(\sqrt{n}), \quad n\to\infty.$$ 
\\
(i) We first bound $P_n$ from above. To do this, note that
\begin{multline} \label{part3_1}
P_n = \mathbf{P}(S_k\ge 0, \, 0<k<n, \, S_{n}=-1) \le 
\sum_{l,r\ge 0} \mathbf{P}(S_{k}\ge 0, \, k\in I_1, \, S_{m_n} = l)\times  
\\ \times 
\mathbf{P}(S_{n-m_n}-S_{m_n}=r-l) \, \mathbf{P}(S_{k} - S_{n} > 0, \, k\in I_3, \, S_{n-m_n} - S_{n} = r+1). 
\end{multline}
Moreover, by the local limit theorem (Theorem 2.3,~\cite{BorLoc})
\begin{multline}
\label{Local}
\mathbf{P}(S_{n-m_n}-S_{m_n}=r-l) = \frac{1}{\sqrt{2\pi (\eta_{n,n-m_n}-\eta_{n,m_n})}}\times\\ \times\left(\exp\left(-\frac{(r-l-\eta_{n,n-m_n}+\eta_{n,m_n}+n-2m_n)^2}{2(\eta_{n,n-m_n}-\eta_{n,m_n})}\right)+o(1)\right),\quad n\to\infty,\nonumber
\end{multline}
with the $o(1)$ uniformly small in $l,r\in \mathbb{N}\cup\{0\}$.  
To apply the theorem, we verify its conditions. The characteristic function of one summand $X_{n,i}$ satisfies
$$
\psi_{n,i} = \exp(\lambda_{n,i} (e^{it}-1) - it),\quad
|\psi_{n,i}|=\exp(\lambda_{n,i}(\cos t - 1)).
$$
Since the sequences $\{\lambda_{n,i},n\ge 1\}$ converge uniformly in $i$ to $1$ as $n\to\infty$, the modulus of these characteristic functions is uniformly bounded away from one for $t\in [\varepsilon, 2\pi-\varepsilon]$. Hence, condition $Z$ of the theorem holds. Conditions (2.3) and (2.4) of that theorem follow from the convergence of $\lambda_{n,i}$ to one. Next, we check condition (UI), which is as follows: for any $\varepsilon$ there exists $M$ such that for all sufficiently large $n$ 
$$
\max_{i,n}\frac{1}{\lambda_{n,i}} {\bf E}(X_{n,i}^2; X_{n,i}>M)\le \varepsilon. 
$$
Since $X_{n,1}$ stochastically dominates $X_{n,i}$ for $i\le n$, we have
$$
\frac{1}{\lambda_{n,i}} {\bf E}(X_{n,i}^2; |X_{n,i}-\lambda_{n,i}|>M)\le \frac{1}{\lambda_{n,n}} {\bf E}(X_{n,1}^2; X_{n,1}>M).
$$
The right-hand side, by the monotone convergence theorem, tends as $n\to\infty$ to 
$$
{\bf E}(X^2; X>M),
$$
where $X\sim Poiss(1)$, and thus can be made arbitrarily small by choosing $M$ sufficiently large. Hence, relation (\ref{Local}) holds.  
Note that the same result can be obtained by a direct application of Stirling's formula, analogous to the reasoning in (\ref{Stirling proof}).  
Furthermore, 
$
\eta_{n,n-m_n} - \eta_{n,m_n} \sim n,\ n\to\infty,
$ since
$$
\frac{\eta_{n, n-m_n} - \eta_{n, m_n}}{n-2m_n} = 
\frac{\lambda_{n,m_n+1}+\dotsb+\lambda_{n,n-m_n}}{n-2m_n}\to 1,\quad n\to\infty,
$$
by the uniform convergence of $\lambda_{n,i}$ to one. Thus, by bounding above the right-hand side of (\ref{part3_1}), we obtain
\begin{equation}\label{c_to0_upper}
P_n \leq \frac{1+o(1)}{2\pi \sqrt{n}} \, \mathbf{P}(S_k\ge 0, \, k\in I_1) \, \mathbf{P}(S_{k}-S_n > 0, \, k \in I_3). 
\end{equation}
\\
Moreover, the probability 
$$
\mathbf{P}(S_k\ge 0, \, k\in I_1)
$$
is bounded above by the corresponding probability for a random walk $\{S_{k,1}^{(u)}\}$ with steps $Poiss(\lambda_{n,1})$ and bounded below by that for a random walk $\{S_{k,1}^{(l)}\}$ with steps $Poiss(\lambda_{n,m_n})$. Here we have used the stochastic domination of $X_{n,i}$ by $X_{n,1}$ and of $X_{n,i}$ by $X_{n,m_n}$ for any $i\in I_1$. However, by Lemma~\ref{left_lemma} the first of these probabilities is equivalent to $2(\lambda_{n,1}-1)$ and the second to $2(\lambda_{n,m_n}-1)$. It remains to note that as $n\to\infty$
\begin{equation}\label{lambdan-1}  
    \begin{gathered}
    \lambda_{n,1}-1 = \frac{c_n}{1-\left(1-c_n/n\right)^{n}} - 1 = \frac{c_n}{c_n - c_n^2/2+o(c_n^2)} - 1  = \frac{c_n}{2} + O(c_n^2), \\[1mm]
\lambda_{n,m_n}-1 = \frac{c_n\left(1-c_n/n\right)^{m_n}}{1-\left(1-c_n/n\right)^{n}} - 1 = \frac{1- m_{n}c_n/n+o(c_n)}{1-c_n/2+o(c_n)} - 1 =\frac{c_n}{2}+ O(c_n^2).
    \end{gathered}
\end{equation}
Hence,
\begin{equation}\label{left_result2}
\mathbf{P}(S_{k}\ge 0, \, k\in I_1)\sim c_n,\quad n\to\infty.
\end{equation}
Similarly, the estimate 
\begin{equation}\label{right_result2}
\mathbf{P}(S_{k} - S_{n} > 0, \, k\in I_3)\sim \frac{1}{2} c_n,\quad n\to\infty,
\end{equation}
is proved analogously using Lemma \ref{right_lemma} and the relations
\begin{equation}
\begin{gathered}
1-\lambda_{n,n} = 1- \frac{c_n\left(1-c_n/n\right)^n}{1-\left(1-c_n/n\right)^{n}}   = \frac{c_n}{2} + O(c_n^2),\quad n\to\infty,\\[1mm]
1-\lambda_{n,n-m_n} = 1-\frac{c_n\left(1-c_n/n\right)^{n-m_n}}{1-\left(1-c_n/n\right)^{n}}   = \frac{c_n}{2}  + O(c_n^2),\quad n\to\infty.
\label{lambdann-1}
\end{gathered}
\end{equation}
Substituting (\ref{left_result2}) and (\ref{right_result2}) into (\ref{c_to0_upper}), we obtain
\begin{equation}\label{c_to0_sup1}
\limsup_{n\to\infty} \frac{2\sqrt{2\pi n}\, P_n}{c_n^2} \le 1.
\end{equation}
\\
(ii) Next, we bound $P_n$ from below 
\begin{multline} \label{2ii}
P_n \geq
\sum_{l,r\le 2 m_n c_n}
\mathbf{P}(S_k\ge 0, \, k \in I_1, \, S_{m_n}=l)\; \mathbf{P}(S_{n-m_n}-S_{m_n}=r-l) \times \\
\mathbf{P}(S_{k} - S_{n} > 0, \, k\in I_3, \, S_{n-m_n} - S_{n} = r+1) - 
\mathbf{P}(\exists i\in I_2: S_{i} = -1, \, S_{n}=-1).
\end{multline}
By Lemma \ref{mid_lemma}, applied with $m = m_n$,
the subtracted term is bounded above by
\begin{equation}
\label{EstPnLower}
\frac{500}{c^2_n \sqrt{2\pi m_n n}} \exp\left(-\frac{m_n c_n^2}{200}\right) = \frac{500\, c_n^{r_n/200-2}}{\sqrt{2\pi m_n n}} = o\left(\frac{c_n^2}{\sqrt{n}}\right),\quad n\to\infty.
\end{equation}
Note that by (\ref{lambdan-1}) and (\ref{lambdann-1})
\begin{eqnarray*}
\eta_{n,m_n}-m_n  = \frac{m_n c_n}{2} + O(m_n c_n^2),\quad n\to\infty, \\[1mm]
n-\eta_{n,n-m_n}-m_n = -\frac{m_n c_n}{2} + O(m_n c_n^2),\quad n\to\infty,
\end{eqnarray*}
whence 
$$
n-2m_n - \eta_{n,n-m_n}+\eta_{n,m_n} = O(m_n c_n^2) = o(\sqrt{n}),\quad n\to\infty. 
$$
Thus, by (\ref{Local}) 
\begin{equation}
\label{LocalLimitr-l}
\mathbf{P}(S_{n-m_n}-S_{m_n}=r-l)=\frac{(1+o(1))}{\sqrt{2\pi n}},\quad n\to\infty,
\end{equation}
with $o(1)$ uniformly small for $l,r\le 2m_n c_n$, since 
$$
(n-2m_n - \eta_{n,n-m_n}+\eta_{n,m_n}+r-l)^2 =  o(n) = o(\eta_{n,n-m_n}-\eta_{n,m_n}),\quad n\to\infty.
$$
Substituting (\ref{LocalLimitr-l}) and (\ref{EstPnLower}) into (\ref{2ii}), we obtain
\begin{equation}
\label{Liminf}
\liminf_{n\to\infty} \frac{2\sqrt{2\pi n}\, P_n}{c_n^2} \ge \liminf_{n\to\infty} \frac{2 Q_{n,1} Q_{n,2}}{c_n^2},
\end{equation}
where 
\begin{eqnarray}
Q_{n,1} = \mathbf{P}(S_{k}\ge 0, \, k \in I_1,\ S_{m_n} \le 2m_n c_n), \label{Qn1}\\[1mm]
Q_{n,2} = \mathbf{P}(S_{k} - S_{n} > 0, \, k \in I_3, \, S_{n-m_n} - S_{n} \le 2m_n c_n). \label{Qn2}
\end{eqnarray}
Moreover, for any positive $h$ we have
\begin{gather*} \label{Smn>2mc}
\mathbf{P}(S_{m_n}>2m_n c_n) \le
e^{-h m_n - 2h m_n c_n} {\bf E}e^{h (S_{m_n}+m_n)} = 
\\[1mm]
= e^{\eta_{n,m_n}\left(e^h-1\right)-h m_n -2h m_n c_n}.
\end{gather*}
Since
\begin{gather*}
\eta_{n,m_n} \left(e^h-1\right) - h m_n = m_n \left(e^{h}-1-h\right) + (1+o(1)) \frac{m_n c_n}{2} \left(e^h-1\right),\quad n\to\infty,
\end{gather*}
by taking $h=c_n$ we obtain
\begin{gather*}
\eta_{n,m_n}\left(e^h-1\right)-h m_n -2h m_n c_n =\\ = m_n c_n^2 \left(\frac{e^{c_n}-1-c_n}{c_n^2}+\frac{(1+o(1)) \left(e^{c_n}-1\right)}{2c_n}-2\right)
\sim r_n \ln c_n ,\quad n\to\infty.
\end{gather*}
Hence, for all sufficiently large $n$
$$
\mathbf{P}(S_{m_n}>m_n c_n) \le e^{2\ln c_n}=o(c_n),\quad n\to\infty.
$$
Therefore, for $c_n\to 0$, $n\to\infty$, we have from (\ref{Qn1})
\begin{equation}\label{finalQn1}
Q_{n,1} = \mathbf{P}(S_{k}\ge 0, \, k \in I_1) + o(c_n) = c_n+o(c_n),
\end{equation}
where in the last step we used (\ref{left_result2}).  
Similarly, using (\ref{right_result2}), we obtain from (\ref{Qn2})
\begin{equation}
\label{finalQn2}
Q_{n,2} = \frac{1}{2}c_n + o(c_n),\quad n\to\infty.
\end{equation}
Using (\ref{finalQn1}) and (\ref{finalQn2}) in (\ref{Liminf}), we obtain
\begin{equation}\label{c_to0_sup2}
\liminf_{n\to\infty} \frac{2 \sqrt{2\pi n}\, P_n}{c_n^2}\ge 1.
\end{equation}
From (\ref{c_to0_sup1}) and (\ref{c_to0_sup2}) it follows that $P_n\sim c_n^2/(2\sqrt{2\pi n})$, hence
\begin{equation}\label{final3}
\mathbf{P}(S_k\ge 0, \, 0<k<n \mid S_{n}=-1) \sim \frac{1}{2} c_n^2,\quad n\to\infty.
\end{equation}
Substituting (\ref{final3}) into (\ref{lem_result_theorem}), we obtain
\begin{equation*}
P_n(p) \sim \frac{1}{2} c_n^2 \left( 1-\left( 1-\frac{c_n}{n}\right)^{n}\right)^{n-1},\quad n\to \infty,
\end{equation*}
thus completing the proof of part 3.

\subsection*{Proof of Case 1}

For the case $c_n \to +\infty,\ n \to \infty$ we show that
$$\mathbf{P}(S_k \geq 0, \ 0< k < n\mid S_{n} = -1) \to 1,\ n\to \infty.$$ 
For any $\varepsilon\in (0,1)$ we can find a parameter $c(\varepsilon)\in (0,+\infty)$ such that 
$$1-\varepsilon = (1-e^{-c})\left(1-\frac{c}{e^c-1}\right)=(1 - e^{-c}(1 + c)).$$ 
This can be done because the function on the right-hand side is continuous and monotonically increasing, taking the value zero at zero and tending to one at infinity. For some natural number $N$ and all $n>N$ the inequality $c_n > c(\varepsilon)$ holds. Consider the random walk $\{\widetilde{S}_k, k\ge 0\}$
with independent steps $\widetilde{X}_i -1 \sim Poiss(\widetilde{\lambda}_{n,i})$, $i\le n$, where the sequence $\{\widetilde{\lambda}_{n,i}\}$ is defined by relation (\ref{lambda}) with $c_n = c(\varepsilon)$.  
By Lemma \ref{StocDomLemm} we have
$$
\mathbf{P}(S_k \geq 0, \ 0< k < n\mid S_{n} = -1) \geq \mathbf{P}(\widetilde{S}_k \geq 0, \ 0< k < n\mid \widetilde{S}_{n} = -1).
$$
Hence, by part 2 of the present theorem
$$
\liminf_{n\to\infty} \mathbf{P}(S_k \geq 0, \ 0< k < n\mid S_{n} = -1) \ge 1-\varepsilon.
$$
Since $\varepsilon$ is arbitrary, it follows that
$$P_n(p) \sim \left( 1-\left( 1-\frac{c_n}{n}\right)^{n}\right)^{n-1},\quad n \to \infty.
$$
\subsection*{Proof of Case 4}
Now, consider the case when $c_n = o(1/n), \ n \to \infty$. Consider the sequence $\widetilde{S}_k = \sum_{i=1}^k  \widetilde{X}_i$, where $\widetilde{X}_i + 1 \sim Poiss(1)$. Then by Lemma~\ref{lemma1}
\begin{equation*}
    \mathbf{P}\left(\widetilde{S}_k \geq 0 , k<n , \widetilde{S}_{n} = -1 \right) = \frac{1}{n}\mathbf{P}\left(\widetilde{S}_{n} = -1 \right).
\end{equation*} 
Apply Lemma \ref{StocDomLemm} to $S_k$ and $\widetilde{S}_{k}$. Since 
\begin{equation*}
\frac{\sum_{j=1}^{i} \lambda_{n,j}}{\sum_{j=1}^{n}\lambda_{n,j}} = \frac{\eta_{n,i}}{ n} = \frac{1-(1-c_n/n)^i}{1-(1-c_n/n)^n}
\ge \frac in,\quad i \le n,
\end{equation*}
we obtain the lower bound
\begin{equation}\label{th4below}
\mathbf{P}\left(\left.S_k \geq 0 , k<n \right|S_{n} = -1 \right) \geq \mathbf{P}\left(\left.\widetilde{S}_k \geq 0 , k<n \right| \widetilde{S}_{n} = -1 \right) =  \frac{1}{n}.
\end{equation}
Consider the random walks $\widehat{S}_k = \sum_{i=1}^k \widehat{X}_i$ and $S_k^* = \sum_{i=1}^k (\widetilde{X}_i +\widehat{X}_i )$, $k\ge 0$, where $\widetilde{X}_i + 1\sim Poiss(1)$ and $\widehat{X}_i\sim Poiss(\lambda_{n,i}/\lambda_{n,n}-1)$, $i\le n$, are independent sequences. Then $\widetilde{X}_i +\widehat{X}_i + 1 \sim  Poiss(\lambda_{n,i}/ \lambda_{n,n})$.
Therefore, by Remark \ref{StocDomRem}
\begin{equation}\label{*equality}
\mathbf{P}\left(\left.S_k \geq 0 , k<n \right| S_{n} = -1 \right) = \mathbf{P}\left(\left.S^*_k \geq 0 , k<n \right| S^*_{n} = -1 \right).
\end{equation}
Note that $\widetilde{S}_n  + n \sim Poiss(n)$, $ \widehat{S}_n \sim Poiss(a_n)$, and $S^*_n + n \sim Poiss(n + a_n)$, where 
\begin{equation}\label{an_small}
a_n = \eta_{n,n}/ \lambda_{n,n} - n \le n \lambda_{n,1} / \lambda_{n,n} - n  \sim n c_n = o(1), \quad n \to \infty.
\end{equation}
Then 
\begin{equation}\label{an_ratio}
\mathbf{P}(S^*_n = -1) / \mathbf{P}(\widetilde{S}_n = -1) = \exp(a_n)\frac{(n+a_n)^{n-1}}{n^{n-1}} = 1 + o(1), \quad n\to\infty.
\end{equation}
Using Lemma \ref{lemma1} for the random walk $i+\widetilde{S}_k$, we obtain 
\begin{multline}\label{an_main}
    \mathbf{P}\left(S^*_k \geq 0 , k<n , S^*_{n} = -1 \right) = \sum_{i=0}^{n} \mathbf{P}\left(S^*_k \geq 0 , k<n , S^*_{n} = -1, \widehat{S}_n = i \right)  \le 
    \\
    \le \sum_{i=0}^{n} \mathbf{P}\left(\widehat{S}_n = i\right) \mathbf{P}\left(i+\widetilde{S}_k \geq 0,\  k<n,\  i+\widetilde{S}_n = -1 \right) = 
    \\
    = \sum_{i=0}^{n} e^{-a_n}\frac{a_n^i}{i!}\frac{i+1}{n}\mathbf{P}(i+\widetilde{S}_n = -1)  
    \leq \frac{e^{-a_n}}{n}\mathbf{P}(\widetilde{S}_n = -1) \left(1+2\sum_{i=1}^{n} a_n^i \right),
\end{multline}
where in the last inequality we used the relation
$$
\mathbf{P}(i+\widetilde{S}_n = -1)\le\mathbf{P}(\widetilde{S}_n = -1),
$$
which holds for all $i\in \{0,1,\dotsc,n\}$. Using (\ref{an_small}) and (\ref{an_ratio}), from (\ref{an_main}) we obtain the upper bound
\begin{equation}\label{th4up}
    \limsup_{n\to\infty} n\,\mathbf{P}\left(\left.S^*_k \geq 0 , k<n \right| S^*_{n} = -1 \right) \leq 1,\quad n \to \infty.
\end{equation}
Using (\ref{th4below}), (\ref{*equality}), and (\ref{th4up}), we deduce
$$
\lim_{n\to\infty} n\,\mathbf{P}\left(S_k \geq 0 , k<n \mid S_{n} = -1 \right) 
 = 1,
$$
from which the required assertion follows.

\end{proof}

\section{Conclusion}\label{conclusion}
In this paper, we propose an approach for analyzing the connectivity probability of an Erdős–Rényi graph based on the theory of inhomogeneous random walks. This method avoids the laborious combinatorial work required for each individual case arising from the dependence of the edge probability on the graph size \(n\).

The method can be applied in a broader range of situations. In future work, we plan to demonstrate how the presented approach can be used to develop a fast method for generating Erdős–Rényi graphs conditioned on connectivity in the sparse regime. This will open up opportunities for more efficient modeling and investigation of the properties of connected random graphs. Furthermore, in subsequent studies, we intend to apply our method to the analysis of random bipartite graphs. It is expected that the developed approach will yield new results in the theory of random graphs and deepen our understanding of their properties.

\section*{Acknowledgments}
The authors thank D.A. Shabanov and A.M. Raigorodskii for their valuable comments and discussions. We are extremely grateful to the anonymous referee for the meticulous work whose remarks allowed us to correct several shortcomings and errors.

\section*{Submission Note}
This is a preprint version of the article. It is under review for publication in \textit{Theory of Probability and its Applications}.

\bibliography{lit_eng}

\end{document}